\documentclass{amsart}

\usepackage{amsmath,amssymb,amsthm}
\usepackage{graphicx}
\usepackage{pdfpages}
\usepackage{verbatim}
\usepackage{tikz}
\usepackage{enumitem}
\usetikzlibrary{arrows,positioning}
\tikzset{
    tile/.style={
           circle,
           draw=black},
}

\newcommand{\uhr}{\upharpoonright}

\newtheorem{defn}{Definition}
\newtheorem{thm}{Theorem}
\newtheorem{prop}{Proposition}
\newtheorem{lemma}{Lemma}

\newcommand{\bdefn}{\begin{defn}}
\newcommand{\edefn}{\end{defn}}
\newcommand{\bthm}{\begin{thm}}
\newcommand{\ethm}{\end{thm}}
\newcommand{\bitem}{\begin{itemize}}
\newcommand{\eitem}{\end{itemize}}
\newcommand{\bpf}{\begin{proof}}
\newcommand{\epf}{\end{proof}}

\newcommand{\on}{\operatorname}
\newcommand{\poly}{\on{poly}}
\newcommand{\Nat}{\mathbb{N}}

\newcommand{\Z}{\mathbb{Z}}

\title{Seas of squares with sizes from a $\Pi^0_1$ set}

\author[L.B.~Westrick]{Linda Brown Westrick} 
\address{School of Mathematics and Statistics, Victoria University of Wellington, Wellington, New Zealand}
 \email{westrick@msor.vuw.ac.nz}

\thanks{The author was supported by Noam Greenberg's Rutherford Discovery Fellowship as a postdoctoral fellow.}

\begin{document}

\begin{abstract}
For each $\Pi^0_1$ $S\subseteq \Nat$, let the $S$-square shift be the 
two-dimensional subshift 
on the alphabet $\{0,1\}$ whose elements consist of squares of 1s of various 
sizes on a background of 0s, where the side length of each square is in $S$.  
Similarly, let the distinct-square shift consist of seas of squares such 
that no two finite squares have the same size.
Extending the self-similar Turing machine tiling construction of \cite{drs}, we 
show that if $X$ is an $S$-square shift or any effectively closed subshift 
of the distinct square shift, then $X$ is sofic.
\end{abstract}

\maketitle

The class of multidimensional sofic shifts contains many accessible examples of 
a geometrical nature. 
However, since the sofic shifts 
properly contain the shifts of finite type (SFTs), they also include shifts
 whose elements encode and make use of arbitrary Turing computations, as
 pioneered in \cite{wang, berger, robinson} and applied in \cite{as, dls, 
drs, hoch, hm, simpson_med} and more.  For definitions, see  
Section \ref{sec:prelim}.

Despite this flexibility, the sofic shifts are still a proper subclass of the
 effectively closed shifts, as evidenced by the mirror shift 
and examples in \cite{pavlov, ru, dls}\footnote{In \cite{ru} 
the existence of effective $\Z^d$ subshifts with only complex patterns was 
established; in \cite{dls} it was shown that a subshift with only complex 
patterns cannot be sofic.}.
  It is informative to 
consider what goes wrong in the following ``construction'' realizing an 
arbitrary effectively closed shift $A$ as a sofic shift.  Given an element $X$ 
which may or may not be in $A$,
\begin{enumerate}
\item Superimpose on $X$ the Turing machine computation which enumerates the 
forbidden patterns of $A$.
\item Whenever a forbidden pattern is enumerated, the computation checks 
whether this pattern has been used in $X$.
\item If the pattern has been used, the computation halts, keeping $X$ out of 
the sofic shift.
\end{enumerate}
The problem is with Step 2.  A Turing computation has only limited access to 
the bits of $X$ on which it is superimposed.  The exact scope of the limitation 
has not been described, but all examples known to the author 
of effectively closed shifts which 
are not sofic were obtained by in some sense 
allowing elements to pack too much important information into a small area. 
On the other hand, 
constructions like the above have been used in \cite{hm, hoch, drs, as}
in the special circumstance where each element of the shift 
is known to have constant columns 
(and in \cite{hm, hoch}, additional geometric properties).  In \cite{hm}, 
such shifts played a role in characterizing the entropies of the 
multidimensional SFTs.  In \cite{drs} and independently \cite{as}, it was 
shown that every effectively closed $\Z^2$ shift consisting of elements 
with constant columns is sofic. 
In this paper, we demonstrate the soficness of a class of shifts with a 
different geometric restriction.

\begin{defn}
For any $S\subseteq \mathbb{N}$, the two-dimensional \emph{$S$-square subshift} 
on the alphabet $\{0,1\}$ is the shift whose elements consist of squares 
of $1$s of various sizes on a background of $0$s, where the side length of each 
square is in $S$. 
\end{defn}

For the sake of being specific let us say that the squares can touch 
diagonally, but this detail does not matter.  Of course, infinite/degenerate 
squares, such as a whole quarter- or half-plane being filled with 1s, are 
unavoidably included when squares of arbitrarily large size are.
The main theorem is

\begin{thm}\label{thm:1}
For any $\Pi^0_1$ set $S$, the $S$-square subshift is sofic.
\end{thm}

The proof expands on the framework developed in \cite{drs}.  In their 
construction, arrays of Turing machines operating at multiple scales work 
together to ``read'' the common row, compute the forbidden words, and 
kill any elements whose common row contains a forbidden word.  The ``reading'' 
process makes strong use of the fact that each bit of the common row is 
repeated infinitely often along an entire column, so the machines have many 
opportunities to work together to access it.

In our construction, a similar array of Turing machines ``measures'' the sizes 
of the squares in a potential element, computes the forbidden sizes, 
and kills the element if it sees a square of a forbidden size.  However, an 
extra challenge is that each machine must measure and store the size of every 
single square in its vicinity, since sizes are not necessarily repeated and 
thus no machine can compensate for another machine overlooking a square.  The 
higher density of information creates some new difficulties which are 
summarized at the start of Section \ref{sec:informal}.

Note that Theorem \ref{thm:1} 
could not be improved to the statement that every effective 
sub-subshift of the $\mathbb{N}$-square shift is sofic.  For example, 
take any effectively closed non-sofic shift $A$ on $\{0,1\}$, 
and embed it into a subshift of 
the $\{1\}$-square shift by taking each element of $A$
and inserting a row and column of $0$s between each original row and 
column in order to separate them.  To avoid this counterexample, we  
consider subshifts consisting of seas of squares in which no size is 
repeated.  
\begin{defn}
The \emph{distinct-square subshift} is the shift whose elements consist 
of squares of $1$s of various sizes on a background of $0$s, where 
the side length of each finite square is distinct.
\end{defn}
With that restriction, similar methods yield the following.
\begin{thm}\label{thm:2}
Every effectively closed subshift of the distinct-square shift is sofic.
\end{thm}

The $S$-square shifts can be thought of as a two-dimensional generalization of 
the one-dimensional $S$-gap shifts.  See \cite{LM} for the definition and basic 
properties, and \cite{HS} for an analysis of some computational properties of 
$S$-gap shifts for $\Pi^0_1$ sets $S$, including the result that some 
right-r.e. numbers are not obtainable as the entropy of a $\Pi^0_1$ $S$-gap 
shift.  However, as is often the case when passing from one dimension to 
multiple dimensions, few of the nice properties of the $S$-gap shift survive 
the dimension increase.  For example, the entropy of the $S$-gap shift for all 
$S$ is well-understood (cf \cite{LM}); by contrast, the study of 
approximations to the entropy of even the $\{1\}$-square shift, 
better known as the hard square 
shift, continues to the present day \cite{P}.  To the 
author's knowledge, nothing is known about the the entropies of the $S$-square 
shifts more generally.  The distinct-square shift has entropy zero.

It is an open question whether 
the following result of \cite{coven} remains true in 
multiple dimensions: every one-dimensional sofic shift is a factor of an SFT of 
the same entropy.  See \cite{desai} for a partial result towards this question 
and \cite{hm} for some discussion.  No counterexample comes out of this class 
of sofic shifts.  At the end of the paper we indicate how to modify our 
construction to produce an SFT with the same entropy as the $S$-square shift 
or distinct-square shift it factors onto.

In Section \ref{sec:expanding} we give a summary of the expanding tileset 
construction 
of \cite{drs} on which our construction builds.  In Section \ref{sec:informal} 
we describe the new algorithm run by tiles at all levels 
to produce an SFT which factors onto a given $S$-square shift, proving Theorem 
\ref{thm:1}.
The supporting Section \ref{sec:algorithm} provides pseudocode of this 
algorithm, for reference and also for 
ease of verifying the claimed runtimes.  
A short Section \ref{sec:distinct} describes how to modify the construction 
of previous two sections to apply to effectively closed subshifts of the 
distinct square shift, proving Theorem \ref{thm:2}.
Finally, in Section \ref{sec:entropy} 
we describe how to modify both constructions so that 
they produce an SFT with the 
same entropy as the subshift they factor onto.

The author would like to thank the anonymous referee, 
whose comments have improved the paper's clarity in several places.  
Most beneficially,
they pointed out where insufficient detail was given in a previous version of 
Section \ref{sec:distinct}.

\section{Preliminaries}\label{sec:prelim}

\subsection{Definitions}

Let $G$ be $\mathbb{Z}^d$ for some positive integer $d$.  In this paper 
almost always $d=2$.  Let $\Sigma$ be a finite
alphabet.  If $D\subseteq G$ is finite, a mapping
$\sigma:D\rightarrow \Sigma$ is called a \emph{pattern}.  If $X \in \Sigma^G$
and $\sigma$ is a pattern with domain $D$, we say $\sigma$ \emph{appears in}
$X$ if there is some translation $g \in G$ so that for all $i \in gD$,
$X(i) = \sigma(g^{-1}i)$.  A subset $A \subseteq \Sigma^G$ is a
\emph{subshift} if there is a set of patterns $P$ such that $A$ consists of
exactly the $X \in \Sigma^G$ that do not contain any pattern from $P$.
It is a \emph{shift of finite type (SFT)} if there is a finite such $P$, and
an \emph{effectively closed shift} if there is a computably enumerable such
$P$.  A subshift $A \subseteq \Sigma^G$ is \emph{sofic} if there is SFT
$B \subseteq \Lambda^G$, where $\Lambda$ is a possibly different
alphabet, and a function $\phi:\Lambda \rightarrow \Sigma$, such that
$A$ is the image of $B$ under $\phi$ (with notation abused to apply
$\phi$ to an element of $\Lambda^G$.)  In this situation we say that 
$B$ \emph{factors onto} $A$.
This was not the original
definition of sofic, 
but it is the most convenient to us.  The set $P$ 
is also called the set of forbidden patterns, and the forbidden patterns 
themselves are sometimes called local constraints.

Let $C$ be a finite set of colors.  A \emph{Wang tileset}, or just
\emph{tileset} is any subset
$T\subseteq C^4$, whose elements are understood as squares with
colored sides, together with the rule that two tiles can be laid
adjacent to each other if and only if they have the same color on the
edge they share.  If it is possible to tile the whole plane with a given
tileset, then the set of ways to do this is an SFT contained in
$T^{\mathbb{Z}^2}$.  The forbidden patterns are simply the $2\times 1$ 
and $1\times 2$ combinations of tiles with non-matching adjacent edges.

\subsection{Multitape machines simulated with tiles}\label{sec:prelim_multitape}

Wang \cite{wang} introduced a tileset, with one or more of the 
tiles designated
as anchor tiles, such that any tiling of the plane that includes an
anchor tile also includes the space-time diagram of a Turing
computation.  The tiles literally correspond to small fragments of
plausible-looking space-time diagram, which when pasted together
without discontinuity, inevitably produce a valid computation.  The
anchor tiles are those displaying the beginning of the tape with the
head in the start position and the machine in the start state.  
The tiles can be made to run any computation, but
if the computation halts, an infinite tiling cannot be completed, as there is
nothing to go in the next row.

We will need a multitape tileset for what follows, so here we describe 
how to modify a standard single-tape tileset to get a 
multitape one.  Let there be 
just one head which reads all the tapes below it.  The head retains its ability 
to move right and left along the bundle of tapes, but it has another possible 
action: instruct the $i$th tape to move to the right/left while the other tapes 
remain fixed.  To do this via a tileset/SFT, at the timestep at 
which the move is triggered, we annotate the entire tape that is to be moved 
with an $L$ or an $R$.  Local constraints ensure that if one part of a tape 
thinks it has $L$ or $R$ on it, then all adjacent parts of that tape 
agree, propagating the signal.  Local 
constraints ensure that the part of the tape right under the head has this 
annotation exactly when the computation calls for it.  Finally, we insist 
that if a certain tape 
at a certain timestep is annotated with the instruction to move, then that tape 
should be shifted in the appropriate direction at the next timestep.  Again, 
this is readily enforced by local checks.

\subsection{Efficient simulation and the runtime-preserving recursion theorem}
In what follows we will design an algorithm 
which depends for its good runtime on a multitape architecture, so we will 
declare that our universal Turing machine has multiple tapes.  The expression
$\varphi_{n,k}(x_1,\dots,x_i)$, where $i\leq k$, denotes the output of this 
universal machine simulating the 
$n$th $k$-tape program, which receives up to $k$ inputs, one input per tape. 
We may abbreviate $\varphi_{n,1}(x)$ as $\varphi_n(x)$, in agreement 
with the standard notation.
Let $RT(n,k,x_1,\dots, x_i)$ denote the runtime of the $n$th $k$-tape machine
on inputs $x_1,\dots x_i$.
If the universal machine simulates a machine with fewer tapes than it has, 
it can always do so with only a constant overhead per step of the simulated 
machine.  To do this, it
uses its tapes to mimic the actions of the simulated machine exactly, guided
by the additional tape on which the program is written, which is also used 
to calculate the next simulated move.

By appeal to the Church-Turing thesis, we will often define an
algorithm by describing it informally (possibly including
implementation details), and then assume we have in hand a machine
index $n$ for our algorithm.  By appeal to the $smn$ theorem\footnote{
The $smn$ theorem that says there is an algorithm for
modifying a machine index to replace some of its inputs with
hard-coded values; see e.g. \cite{soare} for details.} we will often
informally describe an algorithm which has some parameter $e$ meant to
be filled in later, and then assume we have in hand a computable
function $f$ such that $f(e)$ is a machine index for the algorithm
after parameter $e$ is filled in.  

Now we'll observe that a runtime-preserving version of the recursion theorem
holds, via the same proof as the regular recursion theorem.  The runtime 
preservation comes at the small cost of one additional tape.
\begin{prop}
For any computable function $f$ and any $k$,
there is an input $n$ and a constant $C$ such 
that for all inputs $x_1,\dots x_k$, the computations 
$\varphi_{n,k+1}(x_1,\dots,x_k)$ and
 $\varphi_{f(n),k}(x_1,\dots,x_k)$ either both converge 
to the same value or both diverge,
and if they both converge then 
$RT(n,k+1,x_1,\dots,x_k) < 
C \cdot RT(f(n),k, x_1,\dots, x_k).$
 \end{prop}  

\begin{proof}
As in the usual proof of the recursion theorem, we let $n = d(v)$ where $d(u)$ 
is a computable function such that $\varphi_{d(u),k+1}(x_1,\dots, x_{k+1})$ 
does the following:
\begin{enumerate}
\item If $x_{k+1} \neq \emptyset$, halt.
\item On the empty $k+1$st tape, calculate $\varphi_u(u)$ (by any method).
\item Using the $k+1$st tape as the simulation work tape, simulate 
$\varphi_{\varphi_u(u),k}(x_1,\dots,x_k)$ on the other $k$ tapes 
with constant overhead, and return its output.
\end{enumerate}
and $v$ is the index such that $\varphi_v(x) = f(d(x))$.
Then $RT(d(v),k+1,x_1,\dots,x_k)$ 
consists of a constant amount of time for steps (1)-(2), (where step 2 finishes 
because $\varphi_v$ is total) followed by $O(RT(f(d(v)),k,x_1,\dots,x_k))$ 
steps if $\varphi_{f(d(v)),k}(x_1,\dots,x_k)$ converges, and divergence 
otherwise. 
\end{proof}

\section{Expanding tileset constructions}\label{sec:expanding}

To establish notation and conventions for our construction, and because 
we will build on it directly, in this
section we summarize the expanding tileset construction of \cite{drs}.
A tileset $T$ \emph{simulates} a tileset $S$ at zoom
level $N$ if there is an injective map $\phi:S\rightarrow T^{N^2}$
which takes each tile from $S$ to a valid $N\times N$ array of tiles
from $T$, such that
\begin{itemize}
\item For any $S$-tiling $U$, $\phi(U)$ is a $T$-tiling.
\item Any $T$-tiling $W$ can be uniquely divided into an infinite array of
 $N\times N$ macrotiles from the image of $S$.
\item For any $T$-tiling $W$, $\phi^{-1}(W)$ is an $S$-tiling.
\end{itemize}
Here we have abused notation to let $\phi$ map tilings to tilings in
the obvious way.  Given a sufficiently well-behaved sequence $N_0, N_1,\dots$, 
the following construction provides a sequence of tilesets
$T_0,T_1,\dots$ such that for sufficiently large $i$, each $T_i$ simulates 
$T_{i+1}$ at zoom level $N_i$.
Therefore, if $i_0$ is sufficiently large, $T_{i_0}$ simulates each $T_{i_0+k}$
 at zoom level $\prod_{j=i_0}^{i_o+k-1} N_j$.

Let $s$ be the number of bits needed to encode all the side colors for a 
tileset that simulates a universal Turing machine $u$ with four tapes, 
which we refer to as the program tape, the parameter tape, the color tape, 
and the extra tape.\footnote{In this first construction the separate tapes 
are not needed, but we use them so that the main construction can build 
directly on this one.  
The ``extra tape'' is the one which gets swallowed in the 
runtime-preserving recursion theorem.
In \cite{drs}, polynomial overhead 
associated to the recursion theorem was acceptable, so they did not 
need to enter into these details.}
The side colors of 
tileset $T_i$ will be binary strings of length $s+2+2\log N_i$.  
Given a tile, the $4s$ bits obtained by taking the first $s$ bits from each 
color will be called the ``machine part'' of the tile; the next 8 bits (2 per 
color) are the ``wire part'', and the remaining bits are the ``location 
part''.
The permitted location parts are as shown in Figure \ref{fig:drs_layout}(a), 
where addition is mod $N_i$.
\begin{figure}
\begin{tikzpicture}
\draw (.5,1) rectangle (1.5,2);
\node at (1,.5) {$(x,y)$};
\node at (-.25,1.5) {$(x,y)$};
\node at (1, 2.5) {$(x,y+1)$};
\node at (2.75,1.5) {$(x+1,y+1)$};
\fill[color=gray!20] (4, 1.4375) rectangle (4.875, 1.5625);
\fill[color=gray!20] (4.75, .75) rectangle (4.875, 1.4375);
\fill[color=gray!20] (4.875, .75) rectangle (5.125, .125);
\fill[color=gray!20] (5.125,.125) rectangle (5.5625, .375);
\fill[color=gray!20] (5.4375,0) rectangle (5.5625,.125);
\fill[color=gray!20] (5.5625,.25) rectangle (6.75, .375);
\fill[color=gray!20] (6.75,.375) rectangle (6.625, 1.5625);
\fill[color=gray!20] (6.625, 1.5625) rectangle (7, 1.4375);
\fill[color=gray!20] (5.125,.75) rectangle (5.25,2.75);
\fill[color=gray!20] (5.25,2.75) rectangle (5.5625, 2.625);
\fill[color=gray!20] (5.5625, 2.625) rectangle (5.4375,3);
\draw (4,0) rectangle (7,3);
\draw (4.75,.75) rectangle (6.25, 2.25);
\node at (5.5, 1.5) {TM};
\draw (4.75, .75) -- (4.75, 1.4375) -- (4, 1.4375);
\draw (4.875, .75) -- (4.875, 1.5625) -- (4, 1.5625);
\draw (4.875,.75) --(4.875, .125) -- (5.4375,.125) -- (5.4375,0);
\draw (5,.75) -- (5, .25) -- (5.5625,.25) -- (5.5625,0);
\draw (5,.75) -- (5,.25) -- (6.75,.25) -- (6.75, 1.4375) -- (7,1.4375);
\draw (5.125,.75) -- (5.125, .375) -- (6.625,.375) -- (6.625, 1.5625) -- (7,1.5625);
\draw (5.125,.75) -- (5.125, 2.75) -- (5.4375,2.75) -- (5.4275, 3);
\draw (5.25,.75) -- (5.25, 2.625) -- (5.5625, 2.625) -- (5.5625, 3);
\draw (5.5625,3.125) -- (5.5625, 3.25)--(6.6,3.375);
\node at (5.625,3.625) {\footnotesize $s+2+2\log N_{i+1}$};
\draw (5.4375,3.125) -- (5.4375,3.25) -- (4.5,3.375);
\draw (8.25,2) rectangle (9.25,3);
\draw (10, .5) rectangle (11,1.5);
\node at (8,2.5) {$1$};
\node at (8.75, 3.25) {$X$};
\node at (9.5,2.5) {$1$};
\node at (8.75, 1.75) {$X$};
\node at (9.75,1) {$0$};
\node at (10.5, .25) {$0$};
\node at (10.5, 1.75) {$X$};
\node at (11.25,1) {$X$}; 
\node at (1,-.5) {(a)};
\node at (5.5, -.5) {(b)};
\node at (9.5,-.5){(c)};
\end{tikzpicture}
\caption{(a) Location part of a tile with position $(x,y)$. (b) Overall 
layout of an $N_i\times N_i$ macrotile. (c) Tiles with wire parts carrying 
a $1$ horizontally and a $0$ through a turn.}\label{fig:drs_layout}
\end{figure}
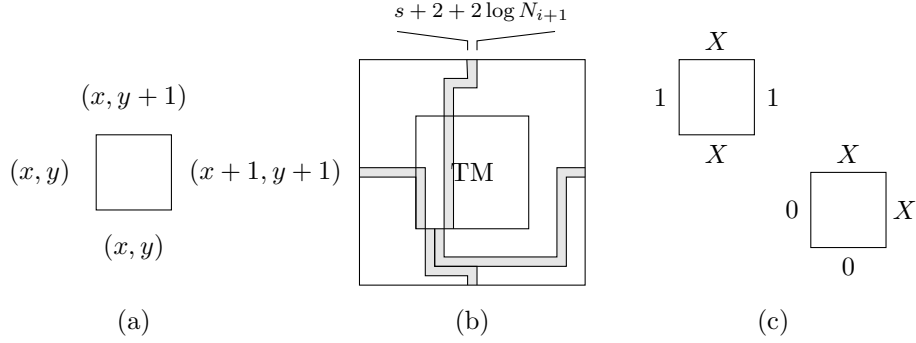
This ensures a unique division of any $T_i$-tiling into an infinite
array of $N_i\times N_i$ macrotiles.  The possibilities for the machine and wire
parts of a tile depend on the location part of the tile, which corresponds to
its position within its $N_i\times N_i$
parent tile.  To ensure a layout such as illustrated in Figure 
\ref{fig:drs_layout}(b), centrally located child tiles must display legal 
universal machine 
tiles with the machine part of their colors; the bottom left tile of this 
computation region must display an anchor tile.

Any tile that falls within 
one of the four $(s+2+2\log N_{i+1})$-tile-wide stripes pictured is
required to use its wire part as in Figure \ref{fig:drs_layout}(c) to 
participate in transmitting a bit between the edge of the macrotile and the 
bottom edge of the computation region.  Any tile at the bottom edge  of the 
computation region acts as the end of the wire, but must verify that 
the data on its wire bits matches what 
is written on the ``color tape'' as displayed by its machine bits.  Any bits  
that are not being used in the ways described should be set to some 
common neutral value.
  
One sees that the $N_i\times N_i$ macrotiles will join with each other
exactly as tiles with macrocolors of the appropriate 
length for a $T_{i+1}$ tile would, 
and the computation region is being set up to do 
computations on those macrocolors.

Let $\varphi_{f(e),2}(i,c)$ be the algorithm which does:
\begin{enumerate}
\item Compute $N_i$.
\item Check $|c| = 4(s+2+2\log N_{i+1})$.
\item Considering $c$ as the concatenation of four colors, check that it obeys 
all the restrictions described above.
\item If the machine part of $c$ gives it a view of any bit of the computation
region's parameter tape, check that bit is consistent with ``$i+1$'' being on 
the parameter tape.
\item As above, check that any visible bit of the program tape is consistent 
with ``$e$'' being the program.
\item\label{step:expanding_last} If any checks fail, halt.  Otherwise enter an 
accepting state and run forever.
\end{enumerate}

Note this algorithm runs in poly$(\log N_{i+1})$ time whenever $i$ is the first 
input, assuming $i\mapsto N_i$ is poly$(\log N_{i+1})$-computable.  
(We didn't specify a particular scheme for the wire layout, but a 
poly$(\log N_{i+1})$ scheme certainly exists.)
Using the 
runtime-preserving recursion theorem, let $n$ be such that 
$\varphi_{n,3}(i,c)$ and $\varphi_{f(n),2}(i,c)$ converge or diverge 
together, and both reach Step \ref{step:expanding_last} in poly$(\log N_{i+1})$ 
time.  Let $R(i) \in O($poly($\log N_{i+1}))$ 
denote the maximum amount of time for the 4-tape 
universal machine simulation $\varphi_{n,3}(i,c)$, where $c$ is arbitrary.
  Assuming poly$(\log N_{i+1}) \ll N_i$, let $i_0$ be large enough that 
$R(i) < N_i/2$ for all $i \geq i_0$.    
For $i\geq i_0$, let
$$T_i = \{c : \varphi_{n,3}(i,c) \text{ enters an accepting state.}\}.$$
It is not difficult to prove that each $T_i$ simulates $T_{i+1}$ 
via the map that sends a $T_{i+1}$-tile $c$ to the unique macrotile whose 
macrocolors are given by $c$.
The key is that for $i\geq i_0$, the computation region is large enough to 
accommodate the entire computation, so a macrotile with given macrocolors $c$ 
is completed if and only if the computation accepts $c$.  (If the computation
halts, there is no way to fill in the computation region, so a macrotile with 
defective colors $c$ cannot be formed.)

We end the section by recalling the assumptions made on $\{N_i\}_{i<\omega}$, 
namely that $i\mapsto N_i$ is poly$(\log N_{i+1})$-computable and that 
poly$(\log N_{i+1}) \ll N_i$.

\section{An SFT that factors onto the $S$-square shift}\label{sec:informal}

In this section we give a proof of the main result for the $S$-square
shift.  As in \cite{drs}, we augment the symbols $0$ and $1$ to
make them also into tiles with various possible color combinations.
The tiles will be designed to carry out the expanding tileset
construction of the previous section, using the extra time in
Step \ref{step:expanding_last} to enumerate $\mathbb{N}\setminus S$.
The higher the layer of simulated macrotile, the more of
$\mathbb{N}\setminus S$ is enumerated, but higher layers are also
further removed from the square size information available at the
pixel level.  To keep the higher level computations aware of the exact
square sizes in their vicinity, the ``child tiles'' making up each
``parent tile'' make sure that the sizes they know about all get
written on their parent's parameter tape.  
With this information, the algorithm running in
a tile can halt its computation, preventing any tiling from being formed, 
if it finds a forbidden size on its parameter tape.
This is the same general
strategy as was used in \cite{drs} for the case of the 
effectively closed shift with constant columns.

For readers familiar with \cite{drs}, the following 
aspects of this construction have no counterpart in that work.  
First, each macrotile 
must keep track of every square in its vicinity, storing 
roughly $L^{2/3}$ bits of information in each macrotile of pixel 
side length $L$.  This is much more than the $\log L$
information that needed to be stored in \cite{drs}.  
This large 
amount of information together with a large required growth rate 
make it impossible for any one 
child to have a copy of its parent's complete information, 
which was the parent-child coordination tool of \cite{drs}. 
We develop a new more distributed communication scheme.
Finally, as the input information takes up a significant fraction 
of the macrotile's tape, polynomial time algorithms no longer suffice;
all our algorithms must run in well under 
quadratic time, necessitating some attention paid to the computation 
model.

The tiles ultimately produced can be defined as the color combinations
accepted by a certain algorithm.  Pseudocode for that algorithm is
given in Section \ref{sec:algorithm} for reference.  Here, we
describe the algorithm informally, motivating each of its components
as we introduce them.

\subsection{Enforcing a sea of squares}

The first step is to observe the $\mathbb{N}$-square shift is sofic,
a fact which is surely folklore, but we include a proof for the reader's
convenience, and because we will build on the SFT $Y$ defined there.

\begin{lemma}
The $\mathbb{N}$-square shift is sofic.
\end{lemma}
\begin{proof}
Consider the alphabet consisting of $0$s and the symbols $\leftarrow,
 \rightarrow \uparrow, \downarrow, \llcorner,  \lrcorner,\ulcorner,
 \urcorner, \circ$.
Let $Y$ be the shift
whose elements look like seas of squares with nested 
directed counterclockwise squares drawn
inside them, plus limit points (see
Figure \ref{fig:Y}).  To see $Y$ is an SFT, we
argue we can obtain it by forbidding every $2\times 2$ pattern that
does not appear in any element of $Y$.

\begin{figure}
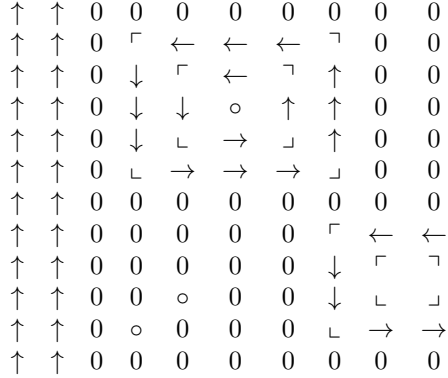

$\begin{array}{cccccccccc}
\uparrow & \uparrow & 0 & 0 & 0 & 0 & 0 & 0 & 0 & 0 \\
\uparrow & \uparrow & 0 & \ulcorner & \leftarrow &\leftarrow & \leftarrow &\urcorner & 0 & 0\\
\uparrow & \uparrow & 0 & \downarrow & \ulcorner & \leftarrow & \urcorner & \uparrow & 0 & 0\\
\uparrow & \uparrow & 0 & \downarrow & \downarrow & \circ & \uparrow & \uparrow & 0 & 0\\
\uparrow & \uparrow & 0 & \downarrow & \llcorner & \rightarrow & \lrcorner & \uparrow & 0 & 0\\
\uparrow & \uparrow & 0 & \llcorner & \rightarrow &\rightarrow & \rightarrow &\lrcorner & 0 & 0\\
\uparrow & \uparrow & 0 & 0 & 0 & 0 & 0 & 0 & 0 & 0 \\
\uparrow & \uparrow & 0 & 0 & 0 & 0 & 0 &\ulcorner &\leftarrow &\leftarrow \\
\uparrow & \uparrow & 0 & 0 & 0 & 0 & 0 &\downarrow &\ulcorner &\urcorner \\
\uparrow & \uparrow & 0 & 0 & \circ & 0 & 0 &\downarrow &\llcorner &\lrcorner \\
\uparrow & \uparrow & 0 & \circ & 0 & 0 & 0 &\llcorner &\rightarrow &\rightarrow \\
\uparrow & \uparrow & 0 & 0 & 0 & 0 & 0 & 0 & 0 & 0 \\
\end{array}$
\caption{A permitted pattern from an element of $Y$}\label{fig:Y}
\end{figure}

These forbidden patterns are enough to enforce that every nonzero symbol
(except for $\circ$) must have its line continued on an adjacent
symbol with the same orientation, 
resulting in directed paths with no beginning or end which make only left turns.
Therefore, each nonzero component is a union of cycles and/or
bi-infinite paths.  We claim first that if a nonzero component contains a
cycle, it contains one of the patterns
$\begin{array}{cc} \ulcorner & \urcorner \\ \llcorner
& \lrcorner\end{array}$ or $\circ$.  Cycles are partially ordered by
spatial inclusion, so pick a cycle which has no cycle inside of it.
Since all turns are left turns, every cycle is convex, so a 
rectangle.  By local restrictions, given a corner, only three things
can go diagonally inside of that corner: another corner of the same
kind (nested), a $\circ$, or a corner of the diagonally opposite kind.
The first case is impossible since it implies another cycle inside
this one.  The latter two cases correspond to the presence of a
$\circ$ and the $2\times 2$ pattern, respectively.

Next, local restrictions guarantee that if a component
contains an $n\times n$ square cycle, either that is the outer
boundary of the component, or the component also contains a concentric
$(n+2)\times(n+2)$ square cycle.  Therefore, if a component has any
cycle, it is made entirely of square cycles nested in the proper way,
and is either a finite or infinite square.

Consider now components with no cycles.  Since all turns are left turns, 
the only possible bi-infinite paths are straight
lines, paths with one corner, and paths with two corners.  Paths with
two corners are impossible by a minimality argument.

If a component contains a path with one corner, there must be another
corner of the same kind diagonally nested in that corner, implying an
entire path with one corner is nested inside the first path.
Therefore, such a component is at least a quarter plane.  Local
restrictions imply that for each path with one corner, either it is
the outer boundary of the component, or there is another
single-corner-path nesting with it on the outside.  Well-formed
quarter-plane or full-plane ``squares'' result.
If component contains a straight line, then by a similar argument it is 
either a half-plane or a full-plane of parallel lines.

Since each nonzero component is a square, this SFT is $Y$.  By replacing 
all nonzero symbols with 1s we obtain the $\mathbb{N}$-square shift.
\end{proof}
Going forward it will be convenient to assume that the alphabet of $Y$ includes 
two versions of each corner symbol, one for use in the interior of squares
and one for use on the outer corners. 
The previous argument is easily adapted to this augmentation.  By 
superimposing tile colors on the symbols of $Y$ instead of the original 
alphabet, we can assume that we are always dealing with a sea of well-formed 
squares, and worry only about their sizes.

\subsection{Growth Considerations}

The general strategy outlined so far has placed some restrictions on a
the possible rates of growth of the expanding tileset sequence
$\{N_i\}_{i<\omega}$ that we will use.  After developing an algorithm,
just as in the previous section we will select some $i_0$ large enough
and form a tileset $T_{i_0}$ to superimpose (according to some
restrictions) on the symbols of $Y$.  Referring to these pixel level
tiles as ``macrotiles of level $i_0$'', we will arrange that for $i>i_0$, a
macrotile at level $i$ occupies a square region whose side length in
pixels is $M_i = \prod_{k=i_0}^{i-1} N_k$.  Let $L_i
= \prod_{k=0}^{i-1}N_k$.  Then $L_i$ is an upper bound on the pixel
size of the macrotiles at level $i$, regardless of what we later
choose $i_0$ to be.

Now consider the information which a macrotile at level $i$ must
``know'' (have written on its parameter tape).  We'll say a square is
within the responsibility zone of a macrotile if it has at least one
full side in the macrotile region.  The maximum number of
distinctly-sized non-overlapping squares that can appear fully with in
a square region of pixel side length $L$ is $O(L^{2/3})$, as one can
see by considering the worst situation, in which exactly one square of
each size $1,2,3,\dots, m$ appears, where $m$ must satisfy
$\sum_{k=1}^m k^2 < L^2$.  Similarly, the maximum number of
distinctly-sized squares that can straddle the boundary of a such a
region is $O(L^{1/2})$.  Encoding the distinct sizes in a standard
binary representation, we will need $O(L_i^{2/3} \log L_i)$ bits
available on the parameter tape of each macrotile at level $i$.  Since
that macrotile is made of an $N_{i-1} \times N_{i-1}$ array of
children, its tape size is $O(N_{i-1})$.  So the sequence
$\{N_i\}_{i<\omega}$ must satisfy
$$ L_i^{2/3}\log L_i \ll N_{i-1},$$
in addition to the poly$(\log N_i) \ll N_{i-1}$ needed 
for the expanding tileset part of the construction.  The reader can verify 
that $N_i = 2^{2^{2^i}}$ is fast enough to satisfy this 
additional constraint, but $N_i = 2^{2^i}$ 
is not.  From here forward we fix $N_i = 2^{2^{2^i}}$.  There will be 
stronger growth requirements for $N_i$ later, but this choice of $N_i$ 
will work for them.

Note, by the definition of $L_i$, it cannot be avoided that $N_{i-1}^{2/3} \ll
L_i^{2/3}\log L_i$.  So whatever algorithm we run to do the
consistency checking must run in (nondeterministic) polynomial time,
but with the degree of the polynomial strictly less than 3/2.    
Some familiar operations which run in linear or $n\log n$ time on
reasonable architectures would require quadratic time on a single tape
Turing machine, which will not be good enough for our purposes.
So we use a multitape machine as described in Section
 \ref{sec:prelim_multitape}.

Let us add a couple details to that implementation which are relevant to the
current construction.  
Since there is an edge to our computation region and tapes can be shifted off 
it, we adopt the convention that any bits that go off the edge are lost and any 
part of the tape that is just returning from beyond the region is blank.  On 
the left edge, we can assume that algorithms are designed for this and don't 
move the tapes backward from their original position.  On the right edge, it 
will not cause a problem because if the head writes a symbol, the distance of 
that symbol from the left edge is less than the time that has elapsed so far 
(since the head had to get to that location), so the distance of that bit from 
the right edge is more than the amount of time that is left.  A perverse 
algorithm could possibly lose some of its input if it devoted all its time to 
shifting the input tape to the right, without even first reading the input.  
But our algorithm reads its input first.

\subsection{Consistency checking algorithm}

We'll now add more parts to the macrocolor scheme (in addition to the
machine, wire, and coordinate parts already being used to create
an expanding tileset) in order to satisfy two goals.
The first goal is
reassuring each individual child that the parent has recorded each
size that appeared in that child.  The second goal is ensuring
that any new large squares, contained in the parent but only
fragmentary in each of the children, have their size recorded in the parent.
The second goal is easier and we address it first.

\subsubsection{Tracking partial squares}

A child may contain partial
squares, such as a corner of a larger square taking a bite from a
corner of the child (there can be at most four such partial
squares) or the side of a larger square taking off the side of a
child (at most two such partial squares).

For any square which has a corner in the child, but does not have a
complete side in the child, let us assume the child knows the
coordinates (in pixels, relative to the child) and orientation of that
corner.
Either the partial square is contained in the parent macrotile or not;
if it is contained, we want to make sure the parent knows about its
size, and if it is not contained we want to make sure the parent
knows its coordinates so that the problem is passed along.  
If a child
has a corner, it uses its knowledge of its own position in the parent
macrotile to compute the deep coordinates of that corner in the parent
macrotile (info of size $O(\log L_{i+1})$).  It then uses a new 
``primary corner message passing part'' of the macrocolor to send that
information to its two neighbor children who it knows also intersect
the same square.  Any child that receives such a message must either use 
it or pass it straight on, stopping only if the edge of the parent tile 
is reached.  
If child with a corner
receives a message in return, it now has two deep parent
corner-coordinates which it can use to calculate the size of
this square, which has been confirmed to be contained in the
parent.  Both children who did this computation 
now require reassurance from the parent about this
size (increasing their number of required reassurances by at most
4 in total; we address the issue of how they get that reassurance below).  

On the other hand, any child with a corner that receives no messages back on
either arm understands that its corner is also a corner in the parent, and
needs to verify that the parent has this corner's coordinates on its parameter 
tape.  To do this we copy a strategy of \cite{drs} and introduce a new
``deep corner copy part'' of the macrocolor.
This part of the macrocolor, size $O(\log L_{i+1})$ bits, is intended to 
display a copy of the part of the parent's parameter tape that contains the 
deep corner coordinates. (Of course, children on the edge of the 
parent macrotile should display some neutral default value to any adjacent 
children of a different parent, displaying the parent information only to 
their siblings.)  To make sure this part of the macrocolor 
says what it should, each child should check that all (usually) four of 
their macrocolors agree on this part.  And any child whose machine part 
lets it see a bit of the parent's parameter tape should check that bit is 
consistent with the deep corner copy part of their macrocolor.  Once these 
checks confirm the accuracy of the information, a child with a corner that 
is also a corner in the parent can directly check 
the deep corner copy part of its own macrocolor 
to confirm that the parent recorded that corner.

Note that an infinite square that covers an entire 
half-plane is invisible to this message system, 
and the macrotiles will not know whether that space is filled 
entirely with 0s or with 1s.  Neither case is forbidden, so 
the macrotile ignorance presents no problem.  
However, in Section \ref{sec:distinct}, we require the 
macrotiles to know where they intersect such partial squares, and 
it is not difficult to augment the message system in order to keep 
track of them.

\subsubsection{Catching the square at a four-way tile system boundary}

The corner message-passing described above will ensure that a
macrotile has on its tape a record of each distinct size of square in its 
responsibility zone. However, in certain exceptional full tilings 
there is the possibility that a single square never enters the 
responsibility zone of any macrotile.  This can happen if the macrotiles at 
every level line up, creating four adjacent pixels which are always
at the corners of their respective macrotiles at every scale.  A square 
containing all four of these pixels could never be recorded. 
To fix this, children located at the
corners of each parent macrotile use a new
``secondary corner message passing part'' of the macrocolor to 
pass deep corner coordinates to
the corner children of adjacent parents. If two such neighbor
children have matching corners, they both compute the relevant square's
size and both require parental reassurance about that size.  If
there is a square at a four-way tile system boundary, it will
eventually lie completely within four macrotiles, who are corners of
their parents, at which point it will be found.

\subsubsection{Reassuring the children}

If we could make a ``parent size list copy part'' of the macrocolor 
and enforce its accuracy just as in the last section, this construction 
would be almost finished, because a child wanting reassurance about a 
particular size could just look for it in that part of their macrocolor.
Here is the problem: the information of the sizes in the parent's 
responsibility zone consumes on the order of $L_{i+1}^{2/3}$ bits, which is much
greater than $N_{i-1}$, the length of the tapes of the children.  So no one
child can hold all the knowledge of its parent.  Instead, each child
will hold some subset of the knowledge of the parent,
nondeterministically choosing which of its siblings to share each
piece of knowledge with, in such a way that every child receives
reassurance that their particular sizes have been recorded by
the parent.

Some children, sitting on the parent tape, have direct access
to a bit of the parent's information.  A group of horizontally adjacent 
children, knowing
that they together hold all the bits describing a single size on
the tape of the parent, can use a new ``parent size reading part'' of the 
macrocolor to coordinate on the correct value for that
size.  So we have a situation where each parent size
is known by at least one child, each child needs to receive
reassurance about up to $L_i^{2/3}$ sizes, and each child can
know about less than $N_{i-1}$ sizes.

So we do make a ``parent size list'' part of the macrocolor, but it carries 
only a subset of the parent's actual list.  
We permit the children to display different parent size lists on
each of their four sides.  
Each size on the list comes 
annotated with a bit to say which direction the information is going.
If a child receives parent information at
one side, it nondeterministically chooses for each of the three other
sides, which of that information is passed on.  
To ensure that parent
information is genuine (not created by children passing a hallucinated
size around in a loop, for example), whenever siblings share a
size they also annotate it with a counter, which the receiving
sibling must increment when passing the size on further.  Only
the children who are part of a group sitting on the part of the parent tape 
showing a
particular size are permitted to use the ``0'' counter for that
size.  Since there are $N_{i}^2$ children, an upper bound on
the space needed for each counter is $2\log N_{i}$ bits.

Since the counter guarantees that received parent information is
genuine, if some size is represented in a child but not in the
parent, there is no way to reassure that child and a valid tiling is
not formed.  On the other hand, if there is some way to distribute all
of the information subject to the constraints, then the children will
nondeterministically find it and produce (perhaps many) valid tilings
in which each child is reassured.

Note that we can ignore the counters
when describing a communication scheme; it suffices that for each parent
size $n$ and each child that needs to know about $n$, there is a
path of children who know about about $n$ which connects the child to
the portion of the parent tape on which $n$ is written.  As long as
all the children who need to know about $n$ are in one connected
component, there is a way to do the counter that is valid, and the
children will nondeterministically find it.

\subsubsection{Abstracting the communication problem}

The communication problem described in the previous section can be
mostly separated from the details of the tiling construction.  Consider a
graph with $N_{i}^2$ nodes arranged in a square grid, where
two nodes are connected by an edge if they are
vertically or horizontally adjacent.  Now suppose that there are
$O(L_{i+1}^{2/3})$ kinds of labels which can be
attached to both edges and nodes.  Each node can receive at most
$O(L_i^{2/3})$ labels.  Furthermore, in
each $M\times M$ subgrid of nodes, there can be at most $O((ML_i)^{2/3})$
distinct labels appearing on nodes in that subgrid.  
Here the nodes are the children, the labels on a node are the sizes about 
which that node needs reassurance, and the $M\times M$ subgrid condition
comes from the bound on the number of distinct sizes of square that
can appear in an $ML_i\times ML_i$ region.  In
this graph context, the problem becomes: is there a way to assign
labels to edges so that all the nodes with a given label are path
connected via edges carrying that label, while limiting the number of
labels per edge to $o(\frac{N_{i-1}}{\log N_{i}})$?  This bound on the
number of labels per edge takes into account the length of the child's
tape $(N_{i-1}/2)$ and the fact that each size consumes $\log L_{i+1} + 2\log
N_{i} \in O(\log N_{i})$ bits of the macrocolor.

In the next section we show that there is always such a labeling.
This labeling connects the children who need reassurances about size $n$ 
to each other, but does not connect them to the parent tape, since the 
information collected by children 
on the parent tape does not satisfy the $M\times M$ subgrid 
condition.  But in the labeling we find, each connected component will include 
an entire row.  So the
children sitting where $n$ is written on the parent tape can
connect to their respective components at a cost of at most one
additional piece of knowledge per child by each passing
their knowledge directly upward or downward.
This completes the description of the algorithm, though we will return to 
its implementation to gain a runtime improvement in Section 
\ref{sec:runtime}.

\subsection{The plaid}\label{sec:plaid}

Returning to the abstracted graph problem of the previous section, 
we now describe a labeling which solves the problem and 
uses $O(2^iL_i^{2/3})$ labels per edge.  The reader can verify that 
$2^iL_i^{2/3}\log N_i \ll N_{i-1}$, so this is acceptable.

The scheme can be visualized as a massively multiscale plaid pattern.
At the top level, there are up to $O(L_{i+1}^{2/3})$ distinct labels
used in the $N_{i}\times N_{i}$ grid.  
Apply labels to all the vertical edges in a pattern of dense
stripes: each edge in the first column receives the first
$L_i^{2/3}$ distinct labels, each edge in the second column receives
the next $L_i^{2/3}$ distinct labels, and so on.  When the
distinct labels run out (as they will since $L_{i+1}^{2/3}/L_i^{2/3} = 
N_i^{2/3}  \ll
N_{i}$) return to the beginning of the list of distinct labels 
and repeat.  Do the same with horizontal
edges.  So far each edge has $L_i^{2/3}$
labels.  This was the first layer of plaid.

Next partition the grid into square subgrids whose
side length is long enough that for every label,
there is a vertical and horizontal line of edges with that label in
each subgrid.  
Such a division creates subgrids with length on the order of $N_{i}^{2/3}$.
By the subgrid condition, the number of distinct labels which have been 
applied to nodes in that subgrid is $O((N_i^{2/3}L_i)^{2/3})$, or 
$O(N_i^{4/9}L_i^{2/3})$.
Now we repeat the process inside each subgrid, adding $L_i^{2/3}$ more 
labels per edge within the subgrid,
this time using
only labels which have been applied to nodes in that subgrid.  
Again this
results in a plaid pattern within the subgrid 
with each label appearing on a horizontal and vertical stripe of edges 
in every subsubgrid with length on the order of $N_i^{4/9}$.
All the edges with a given label remain connected since the new smaller 
scale plaid intersects the larger scale plaid; see Figure \ref{fig:plaid}.

\begin{figure}

\begin{tikzpicture}

\foreach \x in {0, 4, 8}{
\fill[color=gray!20] (\x,0) rectangle (\x+3,3);
\draw[thick] (\x+.8,3)--(\x+.8,0);
\draw[thick] (\x,2.2)--(\x+3,2.2);
}
\node at (1.5,-.5) {(a)};
\node at (5.5, -.5) {(b)};
\node at (9.5,-.5){(c)};

\foreach \x in {4,8}{
\draw[step=.5, xshift=.1cm, yshift =.1cm] (\x-.1,-.1) grid (\x+2.9,2.9);
}

\draw[step=0.075, very thin] (9.5,.5) grid (10,1);
\draw[step=0.075, very thin] (8.5,1.5) grid (9,2);

\end{tikzpicture}
\caption{Edges carrying a fixed label in a square subgrid of side length 
$\sim N_i^{2/3}$.  (a) Labels from the top layer of plaid only (layer 0).
 (b) Labels from layers 0-1. (c) Labels from layers 0-2; we see that of 
the many $\sim N_i^{4/9}\times N_i^{4/9}$ subsubgrids, only two of them 
have nodes which bear the given label.}\label{fig:plaid}
\end{figure}
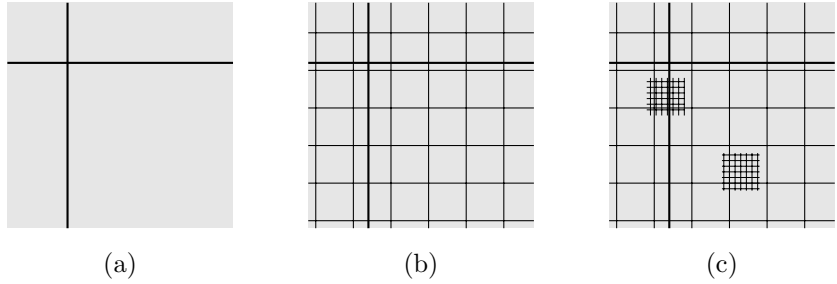

Continuing in this way, we can bound the number of steps it takes to
get down to some fixed small subgrid size (at some point the
implicit rounding in our calculations starts to matter and we take
that point).  The subgrid size at step $j$ is $O(N_{i}^{(2/3)^j})
= O(2^{2^{2^i}(2/3)^j})$, so the number of steps is bounded by
$O(2^{i})$.  Once we reach a fixed small size, we stop with the plaid
and just put all the labels from the nodes in that base grid onto all of 
the edges of that grid, which adds
an additional finite multiple of $L_i^{2/3}$ labels to each edge.  In this way,
all the nodes with a given label are connected
to all other nodes with that label
with a total of $O(2^iL_i^{2/3})$ labels per edge.

\subsection{Runtime considerations}\label{sec:runtime}

Each macrotile at level $i$ has a parameter tape with $O(L_i^{2/3}\log
L_i)$ bits of information on it, the bulk of which is a long list of
sizes of squares contained in that tile.
The color tape of that tile is also filled with macrocolor
information, most of which is $O(2^iL_i^{2/3}\log N_i$) bits worth of
parental sizes, annotated with counters and some indication of the
direction of information flow.

The algorithm the tile runs to determine if the macrocolors behave
appropriately needs to finish in $\ll N_{i-1}$ steps in order for each
level to simulate the next properly and still have time left over for
enumerating $\mathbb{N}\setminus S$.  The bulk of the time is spent
dealing with those very long lists of sizes, so in this section we show
how to make that run close to linear in the number of sizes.  (Recall
that it is necessary to keep the algorithm strictly faster than
polynomial exponent 3/2 in the number of sizes.)  The time taken by
various other steps (expanding tileset, deducing new sizes from
corners) is small in comparison; these little runtime contributions are
collected explicitly in the pseudocode of the next section for
reference.

For each size appearing on the
parameter tape, it must also appear in some macrocolor on the color tape; for
each size that appears on the color tape, all (up to 4) of its macrocolor
appearances must be accompanied by counters that increment
appropriately.  
If all (up to 5) mentions of a particular size could be
physically colocated, then checking that single size for compliance would take
poly$(\log N_i)$ time, so the total time spent checking would be
$2^iL_i^{2/3}$poly$(\log N_i)$, which is in fact the runtime we're aiming for.
But if the mentions are not colocated, even checking a single size
could take $2^iL_i^{2/3}\log N_{i}$ time, most of which is just for the head 
to travel between
mentions.

As a first step to colocation, we modify the universal machine to have 
four color tapes, one
for each macrocolor, so that the macrocolors are essentially
superimposed instead of concatenated.  
We also give the universal machine an additional tape, bringing its total
number of tapes up to 8, so that our algorithm 
can be a 6-tape algorithm that has a work tape and 
treats the color and parameter tapes as read-only.
Next, we declare (and use the
algorithm to check) that the lists of sizes on the parameter tape
and in the macrocolors must be sorted and have no repeat sizes.  
We do not have to do the
sorting ourselves since the nondeterminism can take care of it.  Still
the original problem remains: the lists contain different sizes
and may be badly misaligned.  This can be fixed by allowing the tapes
to move independently.  The point is that now a single instruction can
move $O(2^iL_i^{2/3}\log N_{i})$ many bits at once, with corresponding
runtime savings.

The algorithm looks at the beginning of five superimposed lists (4
macrocolors and one parameter tape), which have been moved so that 
their first entries are aligned, and checks on the least size
$n$ it sees: if it is on the parameter tape does it appear on a
color tape? And do the counters on the color tapes behave
appropriately?  During this time it is ignoring any lists whose first
size is larger than $n$.  If all that checks out, the head moves
forward one data unit, bringing forward with it all the tapes which it
was ignoring on the last step.  This whole combo-step takes time
poly$(\log N_{i})$, and makes the algorithm ready to check the next
smallest size.  After $O(2^i L_i^{2/3})$ repetitions of this,
all the necessary checks have been made.  Thus with a multitape machine, 
the total
runtime for this step is $2^iL_i^{2/3}$(poly$(\log N_{i})$), as desired.

\section{The algorithm for the SFT}\label{sec:algorithm}

The algorithm run by macrotiles at all levels is defined as follows. 
Let $\varphi_{f(e),6}(p, c_1, c_2, c_3, c_4)$ be the algorithm which 
does the following (considering all inputs as strings).

\emph{Initialize:}
\begin{enumerate}
\item Start reading $p$, expecting first two numbers: $i_0, i$ with $i_0\leq i$.
\item Calculate $N_i = 2^{2^{2^i}}$, $L_i = \prod_{j=0}^i N_j$.
\end{enumerate}

\emph{Data Validation:}
\begin{enumerate}[resume]
\item Check that the rest of the data on $p$ is:
	\begin{enumerate}
	\item Up to four deep coordinates, with corner orientation information, 
relative to this tile. Length limit: $8\log(L_i)$ bits
	\item A list of sizes. Length limit:  $O(L_i^{2/3}\log L_{i})$ bits.
(The appropriate constant from the $O$-notation is hard-coded into the 
algorithm.)
	\end{enumerate}
If the length limit is exceeded, halt. Time: $O(L_i^{2/3}\log L_{i})$
\item Check that the data on each $c_k$ is:
	\begin{enumerate}
	\item Machine parts and wire parts. Length: $O(1)$.
	\item Coordinates of this tile in the parent.   Length: $2\log(N_{i})$.
	\item\label{parent-corner-sync} Deep corner copy part: up to four 
deep coordinates, relative to the parent.  Length limit: $8\log L_{i+1}$.
	\item\label{sibling-corner-coordinates} Primary corner message passing 
part.  Up to four deep coordinates, relative to the parent, two incoming and 
two outgoing.  Length limit: $8\log L_{i+1}$
	\item\label{playmate-corner-coordinates} Secondary corner message 
passing part.  Up to two deep coordinates, relative to this tile, one 
incoming and one outgoing.  Length limit: $4\log L_i$
	\item\label{parent-size-sync} Parent size reading part. Length: 
$\log L_{i+1}$.
	\item Parent size list part.  A list of sizes, each less than 
$L_{i+1}$, each size annotated with a counter less than $N_{i}^2$, and a flag 
that says whether that size is incoming or outgoing. Length limit: 
$O(2^iL_i^{2/3}\log L_{i+1})$.
	\end{enumerate}
If the length limit of any part is exceeded, halt.  Time: 
$O(2^i L_i^{2/3} \log L_{i+1})$.
\end{enumerate}

\emph{Dealing with corners:}
\begin{enumerate}[resume]
\item Check the machine, wire and coordinate parts just as in the 
expanding tileset construction.  If the location of this tile lets us 
see a bit of the parent tapes at time 0, check that $i_0$ and $i+1$ are 
consistently on the parameter tape and $e$ is consistently on the program 
tape.  Time: poly$(\log N_{i})$, using the fact that the relevant areas 
on the four macrocolors are colocated.
\item Check that all four deep corner copy parts match (or are neutral valued, 
as appropriate for some sides of a tile on the edge of its parent).  If 
the location of this tile lets us see the deep corner part of the parent 
parameter tape, check consistency with the deep corner copy part. Time: 
poly$(\log L_{i+1})$.
\item If there were any corners on the parameter tape, use the location 
of this tile as well as $M_i = \prod_{j=i_0}^{i-1} N_i$, the pixel length of 
this tile, to calculate the corner's deep coordinates relative 
to the parent.  
Check that the appropriate two macrocolors are displaying these coordinates 
as outgoing messages in the primary corner messaging part 
(unless doing so would display a message outside
the parent tile).  If there are matching incoming coordinates, use them to 
calculate the size of the square, and check that this size appears on one 
of the parent size lists.  Time: poly$(\log N_i)$ for calculations, plus 
$O(2^iL_i^{2/3}\log N_i)$ for looking for the matching length.
\item If this tile is located in a corner location in the parent and also has a 
corner (two different senses of the word ``corner''), and if the orientation 
is such that both arms of the corner exit the parent, repeat the previous 
step with appropriate variation to communicate with adjacent non-sibling 
children through the secondary corner messaging part. Time: 
$O(2^iL_i^{2/3}\log N_i)$.
\item If a corner in this tile has all its messages unanswered, make sure 
that corner appears on the deep corner copy part.  Time: poly$(\log L_{i+1})$.
\item If there is an incoming message in the primary corner message passing 
part of one macrocolor, and it does not match any corner this tile has, 
check the same message is outgoing in the same channel on the opposite 
side (unless doing so would cause the message to display on the outside 
of the parent tile).  If there is an incoming message in the secondary 
corner message passing part with no matching corner in this tile, it is 
ignored.  Time: $O(\log L_{i+1})$.
\end{enumerate}

\emph{Dealing with reassurances:}
\begin{enumerate}[resume]
\item If the location of this tile lets us see a bit of the size list of 
the parent parameter tape, check that the parent size reading parts are 
displaying the same thing on the left and the right (or possibly just 
displaying on one side, if this tile is located at one edge of a parent 
size on the tape).  Check that the parent size being displayed is 
consistent with what is written on the tape.  If the location of this 
tile does not let us see the parent tape, the parent size reading parts 
should be empty.  Time: poly$(\log L_{i+1})$.  
\item Check that the size lists on both $p$ and each macrocolor $c_k$ are 
sorted.  Runtime: $O(2^iL_i^{2/3}\log(L_{i+1}))$.
\item If any size from the parent size list of any macrocolor $c_k$ is 
shown as outgoing with counter 0, check that this size also appears in 
the parent size reading part. Runtime: $O(2^iL_i^{2/3}(\log L_{i+1}))$, 
most of which is travel, because there can be at most four uses of 
counter 0.  
\item For all sizes listed in any macrocolor $c_k$, or on the tape $p$, 
do the following:
	\begin{enumerate}
	\item If the size is listed in a macrocolor, verify that there is 
exactly one macrocolor on which that size is incoming, and that the 
counter annotating the size on any outgoing sides is one more than 
the incoming counter.  Exception: it is ok to have no incoming size if 
all outgoing sizes are annotated with 0.  
	\item If the size is listed on the parameter tape $p$, make sure 
it is listed in some macrocolor.
	\end{enumerate}
Time:  $2^iL_i^{2/3}$(poly$(\log L_{i+1})$), assuming a multitape machine.
\end{enumerate}

\emph{Killing tiles with forbidden sizes:}
\begin{enumerate}[resume]
\item\label{run-free} Having successfully completed all these checks 
(halting to kill the tiling if they fail), run the algorithm which 
enumerates the complement of $S$.  Whenever a size is enumerated there, 
check to see if it appears on $p$'s list.  If it does, halt to kill the 
tiling.
\end{enumerate}

By the runtime-preserving recursion theorem, let $n$ be a fixed point 
such that $RT(f(n),6,p,c_1,\dots,c_4)$, $RT(n,7,p,c_1,\dots, c_4)$, and  
the universal simulation of the same coincide to within a constant factor.
Let $RT(i)$ denote the maximum time it could take for the simulation 
to reach Step \ref{run-free} given that $p$ starts with $i_0,i$ for 
any $i_0\leq i$.
We see that $RT(i)$ is $2^iL_i^{2/3}$poly$(\log N_i)$.  Fix $i_0$ 
large enough that $RT(i) \ll N_{i-1}$ for all $i\geq i_0$.

Now define an SFT $B$ as follows.  For any symbol $s$ from the alphabet 
of $Y$, define $p_s$ as follows.  If $s$ is not an outer corner, $p_s$ 
is $i_0,i_0$.  If $s$ is an outer corner, $p_s$ is 
$i_0, i_0, \langle (0,0), r\rangle$, where $r$ indicates the 
orientation of $s$.
The alphabet of $B$ consists of those
combinations $s, c_1, c_2, c_3, c_4$, where $s$ is a symbol of $Y$,
such that $\varphi_{n,7}(p_s, c_1, c_2, c_3, c_4)$ makes it to
step \ref{run-free} without entering the kill state.
The
forbidden patterns of $B$, of course, are the forbidden
patterns of $Y$, together with the tiling restrictions declaring that
adjacent colors must match.

This completes the pseudo-code for the construction.  The main points 
of the verification are convincing oneself 
that (1) each permitted pattern of the $S$-square shift has a 
valid $B$-preimage, (2) whenever a $Y$-pattern has a valid $B$-preimage, 
each macrotile of the $B$-preimage has on its parameter tape a list 
of all sizes of square within its responsibility zone (including sufficiently 
small squares intersecting at a corner, if they are also contained in the
$Y$-pattern), and (3) for each $n$, there 
is a macrotile size large enough that no macrotile that large 
with a forbidden size less than $n$ on its parameter tape can be formed. 
The details of the verification are numerous but not difficult.

\section{Effectively closed subshifts of the distinct-square shift}\label{sec:distinct}

The previous argument needs only slight modifications in order to 
prove Theorem \ref{thm:2}, that any effectively closed subshift of 
the distinct-square shift is sofic.  We slightly increase the 
information monitored by each macrotile, to ensure that
each macrotile contains a complete 
description of what is going on in its responsibility zone as 
defined below.  Then as 
it uses its extra time to enumerate an arbitrary set of additional 
forbidden patterns, it can check to see whether each pattern appears in its 
responsibility zone, and kill the tiling if so.  In Section \ref{sec:5.2},
we describe how to do this check efficiently.

\subsection{Adjusting each macrotile to keep track of slightly more information}
First, require that the 
parameter tape of each macrotile, in addition to storing the list of 
distinct sizes appearing within its responsibility zone, must also 
annotate each size with the deep coordinates of the lower left hand 
corner of the unique square with that size.  (These deep coordinates 
may lie outside the macrotile for some squares, which is fine.) 
These annotations are passed among the children, and any child 
needing reassurance about a particular size also needs reassurance 
about the exact location of the square with that size.  Since 
the tape is only permitted to hold one set of deep coordinates 
per distinct size, this prevents repeated sizes from occurring within 
the responsibility zone of a given macrotile.  To prevent the parent 
from hallucinating additional sizes, we additionally require the 
parent to receive confirmation that each of its recorded sizes is 
actually realized in a specific child, for example by dictating that 
when a child receives distributed parental information, 
it must either use it or pass it on. 
Lastly, require each macrotile to also keep track of the location of 
any partial sides within its responsibility zone.  By a partial side, 
we mean a row or column of 1s whose corners are not within the 
responsibility zone, such as what appears on the left edge in Figure 
\ref{fig:Y}. The amount of information now needed 
on the tape at level $i$ is essentially unchanged,
still $O(L_i^{2/3}\log L_{i})$.

Second, the responsibility zones must be expanded so that any  
finite region eventually lies within the responsibility zone of a 
single macrotile.  Require each child tile that is on the boundary 
of its parent macrotile to share its parameter tape with its neighbor(s) 
of a different parent via a new part of the macrocolor, and to require 
reassurance from its own parent about any information it receives 
in return.  Children at the corners of four distinct parent tiles 
cooperate to share all their information with each other and 
require additional reassurance about all of it.  In this way, 
the responsibility 
zone of a given macrotile extends beyond its area by at least the 
width of one of its child tiles in every direction, but not more 
than double the width of one of its child tiles.  The result: even in an 
exceptional tiling, every finite region is eventually within 
the responsibility zone of a single macrotile.  The  
amount of information needed on the parameter tapes and macrocolors is 
increased by at most a constant factor.
Note that a macrotile is now tracking deep partial corner coordinates
for corners located in its entire responsibility zone, including 
possibly corners that are outside that macrotile, and same holds for 
coordinates of partial sides.  Again, the method 
of encoding the deep coordinates can be extended to accommodate this. 
The macrotile's tape now contains a complete description of what is 
in its responsibility zone.

\subsection{Recognizing when a given forbidden pattern appears}\label{sec:5.2}
Without loss of generality, we may assume that the enumerated forbidden 
patterns all have square domain.  We may also assume that each 
forbidden pattern has at least one 0 and at least one 1, since if a 
solid square of 0s or 1s is forbidden, there is one with minimal size, 
and it can be forbidden manually.  Since the tileset already forbids 
patterns that do not appear in the $\Nat$-square shift, we may assume 
all forbidden patterns consist of squares and partial squares of 1s. 
For each forbidden pattern, there are just countably many ways to 
complete its partial squares, so we can assume each forbidden pattern is 
enumerated infinitely many times, each time annotated with a particular 
choice of size and location (relative to the domain of 
the forbidden pattern) of each of its partial squares.  To account for 
squares of infinite size, ``Large'' is a 
permitted size, and ``Far'' is permitted coordinate information for 
``Large'' squares.  Observe that if a pattern is forbidden, it is 
correct to forbid all possible ways of completing it, and conversely,
if all ways of completing it are forbidden, the pattern itself cannot 
appear.

It is efficient to recognize whether an annotated pattern appears in 
the responsibility zone.   
Assuming first that the annotated pattern contains at least one 
finite square or partial corner, observe that if that component is 
located in the parameter list, its location in the macrotile 
uniquely specifies a single candidate for the
location of the forbidden pattern relative the macrotile. 
So make one pass through the parameter tape to look for such a component.
If the location-fixing component is not found, or if its position implies the 
location of the forbidden pattern extends outside the
responsibility zone, or if the annotated pattern references finite sizes 
that are larger than what the macrotile tracks, the annotated pattern 
is safely ignored (a larger macrotile will consider it).   
Otherwise, make a second pass through the parameter 
list, this time asking each component whether it intersects the 
candidate region, and if so, whether the 
annotated forbidden pattern thinks it should be there.  
At the end of the second 
pass, if no extraneous components have intruded on the candidate region, 
and if all components of the annotated pattern have been checked off, 
then the forbidden pattern appears and the macrotile is forbidden. 
For a macrotile at level $i$, these checks take time 
$rL_i^{2/3} \poly(\log L_i)$, where $r$ is the 
size of the annotated pattern. 

Suppose now that the annotated pattern contains no square of finite 
size, nor a partial corner.  The only patterns that can do 
this have one or two ``Large'' partial sides, with all their coordinates 
``Far'', and no other squares.  We assume such patterns are split 
into further annotated versions, each specifying the relative 
location of some additional finite square or partial corner, plus 
one annotation asserting that any such square or corner is ``Far''.  
The former case is then handled as described above, and in the latter 
case the macrotile considers the annotated pattern found if it has 
matching partial side(s) and no finite sizes or partial corners on its tape.

The second pass ensures that in every case in which a macrotile is killed, 
it was 
appropriate to do so.  Conversely, if a configuration contains a given 
forbidden pattern, it contains it in a context which is covered by one 
of the annotations, so that configuration will be killed if that 
annotation is ever considered by a macrotile which contains 
both the forbidden pattern and enough of the context.  
And each annotation will be eventually considered, 
because the time to get through the $k$th annotation in a macrotile at 
level $i$ is 
$c_k + (\sum_{j\leq k} r_j) L_i^{2/3}\poly(\log L_i)$, where $c_k$ is 
the computing time to enumerate the 
first $k$ annotations, and $r_j$ is the size of the $j$th annotation.

This completes the sketch of modifications needed to prove Theorem \ref{thm:2}.

\section{An SFT with the same entropy as its factor}\label{sec:entropy}

For any $\Z^2$ subshift $A$, let $A\uhr n$ denote the set of patterns on 
an $n\times n$ rectangular domain $D \subseteq \Z^2$ which appear in 
some element of $A$.  The \emph{entropy} of $A$ is defined as
$$h(A) = \lim_{n\rightarrow\infty} \frac{\log |A\uhr n|}{n^2}.$$

The SFTs described in the previous sections had more entropy than
the shifts they factored onto because in each macrotile, the
children had many choices about how to route the information of
parental sizes.  Also, nothing prevented a macrotile from
believing it had sizes and/or partial corners which it did not 
have.  We show how to
modify the construction to force the tiles to claim only the
sizes they must, and force the internal communication of the
children to follow a specific and unique plaid protocol.  In this
modification, there are no counters; hallucinatory parental
reassurances are made impossible by the rigid nature of the protocol.

Let $A$ be an $S$-square shift or an effectively closed subshift of 
the distinct square shift.
Let $B'$ denote the modified SFT which we will describe below,
with $\Lambda'$ its alphabet.  
Let $F_n \subseteq \Z^2$ be any $n\times n$ square. Let $\partial F_n$
be its boundary.  Let 
$M_i$ denote the side lengths of the macrotiles at
level $i$. We will arrange the following unique extension property:
for each $M_i$, and each pattern 
$\sigma:F_{M_i} \rightarrow \{0,1\}$ that appears in an element of $A$, 
and each pattern 
$\tau:\partial F_{M_i} \rightarrow \Lambda'$ that is consistent 
with $\sigma$, there is at most one way to extend $\tau$ to 
all of $F_{M_i}$ so that the result is a single level $i$
macrotile consistent 
with $\sigma$.  This is enough to ensure that $h(B') = h(A)$, because
$$|B'\uhr M_i| \leq  M_{i-1}^2 |\Lambda'|^{4M_iM_{i-1} + 4N_{i-1}^2M_{i-1}}
  |A\uhr M_i|,$$
where the factor of $M_{i-1}^2$ accounts for all the possibilities 
of how the $(i-1)$th layer of macrotiles can be aligned, and the 
$|\Lambda'|^{4M_iM_{i-1} + 4N_{i-1}^2M_{i-1}}$ factor accounts for all 
the possible ways to fill in symbols of $\Lambda'$ completely 
on partial layer-$(i-1)$ tiles at the edges of the region and 
on the boundary of layer-$(i-1)$ tiles fully contained within the 
region.  Taking logs, dividing by $M_i^2$, and taking the limit as 
$i\rightarrow\infty$ completes the entropy calculation.  Now we 
show how to satisfy the unique extension property.

The partial corner lists are easy to make unique:  we can restrict the partial 
corner part of the parameter tape
to appear in lexicographic order; also, when a child sees its 
deep corner copy 
part include a corner at that child's location, that child must verify 
that it truly has an unrequited corner there.  

To pin down the protocol of parent size list communication among children, 
fix in advance a division of the parent macrotile into $O(2^i)$ layers
of subgrids large enough to accommodate the worst case of the plaid described in
Section \ref{sec:plaid}.  In this way each child tile can know its
location in each subgrid to which it belongs.  The 
macrocolors, instead of containing one long list of parent
sizes, will now contain $O(2^i)$ sorted lists, one associated to
each layer of plaid.  The child tiles will perform checks on these 
lists in order to force a plaid protocol similar to that of 
Section \ref{sec:plaid}, with a few modifications to make it 
possible to enforce that a size is used in a particular subgrid 
if and only if some child in that subgrid needs it.

\subsection{Forcing a plaid pattern at each layer}

To enforce a local stripes pattern, children verify that within each
subgrid to which they belong, the sizes associated to that
region are propagated straight horizontally and vertically within the subgrid,
and not propagated outside of it.  
To ensure that the same stripes are repeated horizontally and vertically 
within a given subgrid at layer $j$, the children who lie on the main 
diagonals of its subgrids at layer $j+1$ can check that their horizontal 
and vertical $j$th lists are in exact agreement.  
To enforce the order of sizes in the plaid, the children just below all 
the main diagonals of the layer $j+1$ subgrids verify that their layer $j$ 
horizontal sizes must all be larger than their layer $j$ vertical sizes.
Since the number 
of sizes used in the given layer $j$ subgrid may be much less than 
than $L_i^{2/3}$ times the width of the $(j+1)$st layer, 
we allow a filler symbol in the size lists.  By insisting that all 
lists be full, and that layer $j$ sizes never follow filler symbols 
within layer $j+1$ subgrids, we obtain well-formed plaids in 
each subgrid, with each size occupying exactly one row and one column 
in each subsubgrid.  
The different possibilities now correspond only to 
different choices of what sizes should go in the plaid in each subgrid.

\subsection{Making exactly the right sizes show up in each subgrid}

We use the ordering of the plaid to ensure that the sizes 
that are listed in any
subgrid of layer $j+1$ are a subset of those in its parent subgrid 
in layer $j$. To check this, the children
compare their $(j+1)$st list to their 
$j$th list, and if some $n$ on the $(j+1)$st list lies between the
least and greatest elements of the $j$th list (or between $0$ and its 
greatest element, for children in the upper left corner of a layer $j+1$ 
subgrid), then $n$ needs to be
on the $j$th list or the child kills the tiling, since $n$
cannot legally appear in the parent subgrid anywhere else.  
This leaves an edge case where a size bigger than the largest 
size of one layer $j$ list and smaller than the smallest size of 
the next could hallucinate itself into the $(j+1)$st layer.
To fix this, we modify the plaid convention to insist
that the size lists in neighboring columns and rows within each 
subgrid are not disjoint,
but overlap by exactly one entry, which is the least element of one
list and the greatest element of the other.
This ensures that
no new sizes are added when passing from layer to layer.

To ensure a size appears in a subgrid only if it needs to be there,
each size that appears in a given subgrid must be annotated with
the coordinates of the lexicographically least child in that subgrid
who needs that size.  So for example, in the topmost layer of plaid,
the relevant region is the entire parent macrotile, and so each size
represented in the top layer of plaid is annotated with the least
child in the entire macrotile who needs that size.  

First the children need to enforce that if a particular size $n$
appears in one layer of plaid with annotation $a$, this size-annotation
pair must be propagated into the subgrid in which $a$ is located in the 
next layer.  To do this, for each $\langle n, a\rangle$ on a given 
child's $j$th list, if that child is located in the same $(j+1)$st-layer 
subgrid as $a$, and if $n$ is between the least and greatest elements 
of the child's $(j+1)$st list, then $\langle n, a\rangle$ must appear 
on the $(j+1)$st list.  If that child is not located in the same $(j+1)$st 
layer subgrid as $a$, but $n$ still appears on its $(j+1)$st list, 
it must appear as $\langle n, b\rangle$ for some $b$ lexicographically 
greater than $a$ which is in the same $(j+1)$st subgrid as this child.
(If everything in the child's $(j+1)$st subgrid is lexicographically 
smaller than $a$, this is a ban on $n$ propagating here.)

Finally, we need to check that all claimed annotations are accurate. 
If the child with location $a$ has $\langle n,a\rangle$ on its last list,
it must verify that it actually needs
$n$ or kill the tiling.  And each child lexicographically less than $a$
which sees $\langle n,a\rangle$ on its last list must verify that it 
does not need $n$, or kill the
tiling.

As before, there is a algorithm for doing all these checks that is 
linear in the number of sizes.

\subsection{Connecting the top layer to the parent tape}
The restrictions ensuring a unique plaid have also ensured that the
top layer of plaid contains exactly the sizes needed by the children.
To make sure the contents of the parent tape correspond exactly to the 
contents of the top layer of plaid, a new part of the macrocolor 
is employed, which holds a size-annotation pair.
The first child of each parent size reading group must use this channel 
to pass the size it has read straight down, along with an annotation 
indicating exactly at what location the message will find that size in 
the top layer of plaid.  Any child receiving such a message must pass 
it straight onward, or verify it, if that child is the one named in 
the annotation.  This ensures everything on the parent tape makes it into 
the plaid.  To ensure everything on the plaid came from the parent tape, 
declare that every size in the top layer of plaid must carry an additional 
annotation, namely the same location where the message from the parent tape 
will meet the plaid; the child at that location can check that the 
message from the tape really arrives.

\subsection{Unique extension property}

To see that the unique extension property holds, one can argue inductively 
that the following stronger property holds at all levels: 
for any valid $A$-pattern $\sigma$ on 
$F_{M_i}$ and any consistent $B'$-pattern $\tau$ on its boundary, not 
only is there at most one consistent extension of $\tau$ to all of $F_{M_i}$
that produces a single level $i$ macrotile, 
but also that when such a macrotile exists, its parameter tape must contain 
exactly the sizes that appear in the macrotile's responsibility zone, 
as well as exactly the deep coordinates of corners of partial squares which 
extend outside the responsibility zone.  Since the responsibility zone 
of the macrotile extends outside the $A$-pattern, this may at first glance 
appear not well-defined.  However, if $\tau$ can be extended to form a 
single macrotile on $\sigma$, then the information $\tau$ displays outside 
the macrotile can be interpreted as message-passing which uniquely 
determines what is going on in the responsibility zone, in the sense 
that there is only one extension of $\sigma$ to the entire responsibility 
zone which allows $\tau$ to also be extended there.

Assuming the stronger property holds at level $i$, its proof sketch for level 
$i+1$ goes as follows.  Given $\sigma$, divide it into $N_{i-1}^2$ child tiles 
of side length $M_{i-1}$.  For each child tile, its coordinates in the 
parent are determined.
Using bits of $\sigma$ for interior edges and 
messages from $\tau$ for exterior edges, determine what is going on in 
each child tile's responsibility zone.  Assuming each child tile will be 
completed, its parameter tape is thus determined.  The corner messages it 
sends out are uniquely determined by its parameter tape, and those it 
receives are uniquely determined by its adjacent sibling's parameter tape 
or by $\tau$.  The level $i$ tile's deep corner coordinates are therefore
uniquely determined as the only ones that all the child tiles will accept, 
namely the ``correct'' ones.  Now that the corner calculations have been 
established, it is also pinned down which children will ask for 
reassurances about what sizes.  The unique plaid protocol uniquely 
determines what will go in the parent size list part of each child's 
macrocolor in order to provide those reassurances, as well as guarantees 
that the parent tape size list part contains exactly the sizes requested 
by the children.  The parent parameter tape is now completely fixed, 
when you count the fact that the algorithm running on it will kill any 
tiling in which it is not appropriately ordered.  This also fixes the 
parent size reading parts of each child tile.  Since the inputs to 
the color tapes of the parent come from $\tau$ and the computation 
is deterministic, the wire and machine parts of all the children 
are also uniquely determined.  So every child's macrocolors are uniquely 
determined.  This completely determines what symbols of $\Lambda'$ 
can go on the boundary of each child tile, because the only messages that 
leave the child tile (at any level of its grandchildren) are secondary 
corner passing messages or responsibility zone information sharing 
messages, the content of which is uniquely determined from $\sigma$ and 
$\tau$.  By induction, the rest of the symbols of $\Lambda'$ are 
uniquely determined.  If any stage of this process was impossible to 
carry out, then there is no extension of $\tau$ to produce all of $\sigma$. 
So the number of possible $\tau$-extensions is either 0 or 1.

\bibliographystyle{plain}
\bibliography{ssquare_bib}

\end{document}